\documentclass[11pt]{amsart}
\usepackage{amsfonts,latexsym,rawfonts,amsmath,amssymb,amsthm}
\usepackage[plainpages=false]{hyperref}  
\usepackage{verbatim}                    
\usepackage{color}

\newtheorem{theorem}{Theorem}[section]
\newtheorem{lemma}[theorem]{Lemma}
\newtheorem{corollary}[theorem]{Corollary}

\theoremstyle{definition}
\newtheorem{definition}[theorem]{Definition}

\theoremstyle{remark}
\newtheorem{remark}[theorem]{Remark}

\numberwithin{equation}{section}

\begin{document}

\title[
]
{Asymptotic behavior at infinity of solutions of Monge-Amp\`ere equations in half spaces}

\author[
]
{Xiaobiao Jia,  Dongsheng Li \and Zhisu Li}
\address{School of Mathematics and Statistics\\	Xi'an Jiaotong University\\	Xi'an 710049, China.}
\email{xiaobiaojia@126.com (Xiaobiao Jia: corresponding author)}
\email{lidsh@mail.xjtu.edu.cn (Dongsheng Li)}
\thanks{The first and the second authors were supported by National Science Foundation of China (Grant No. 11671316) and  the third author was supported by National Science Foundation of China (Grant No. 11801015) and China Postdoctoral Science Foundation (Grant No. 2018M631230).}




\address{Beijing International Center for Mathematical Research, Peking University, Beijing 100871,China.
}
\email{lizhisu@bicmr.pku.edu.cn (Zhisu Li)}





\begin{abstract}
It is proved that any convex viscosity solution
of $\det D^2u=1$ outside a bounded domain of the half space
is asymptotic to a quadratic polynomial at infinity under reasonable assumptions, where the asymptotic rate is the same as the Poisson kernel of the half space. Consequently, it follows the Liouville type theorem on Monge-Amp\`ere equation in the half space.
Meanwhile, it is established the existence theorem for the Dirichlet problem with prescribed asymptotic behavior at infinity.
\end{abstract}

\maketitle

\section{Introduction}

In this paper we investigate the asymptotic behavior at infinity
of convex viscosity solution of the Monge-Amp\`ere equation
\begin{equation}\label{SZ_Main_TM_Equation}
         \left\{
               \begin{aligned}
                      &\det D^2u=f  \quad\mbox{in } \mathbb{R}^n_+,\\
                      &u=p(x)\quad\quad\;\,\mbox{on }  \{x_n=0\},
                \end{aligned}
         \right.
\end{equation}
where the space dimension $n\geq 2$,
$p(x)$ is a quadratic polynomial with $D^2p>0$ and
\begin{equation}\label{SZ_supp_f-1_is_Bounded}
    \Omega_0:=\mbox{support} (f-1)\subset B_{R_0}^+
\end{equation}
for some $R_0>0$.

The classical J$\ddot{\mbox{o}}$gens-Calabi-Pogorelov theorem
(cf. \cite{Jorgens-1954-MathAnn} for $n=2$, \cite{Calabi-1958-MichiganMathJ} for
$n\leq 5$, \cite{Pogoralov-1972-GeometriaeDedicata} for $n\geq2$)
states that any classical convex solution of
\begin{equation*}\label{EQ_det=1}
          \det D^2u=1\quad\mbox{in}\;\mathbb{R}^n
\end{equation*}
is a quadratic polynomial.
In \cite{Caffarelli-1995},
L. A. Caffarelli extended above result to viscosity solutions.
The asymptotic behavior at infinity of viscosity solution of $\det D^2u=1$
outside a bounded subset of $\mathbb{R}^n$ was obtained
by L. A. Caffarelli and Y. Y. Li in \cite{Caffarelli-Liyanyan-2003-CommPureApplMath}, where the
main conclusion is that
for $n\geq3$, $u$ tends to a quadratic polynomial at infinity with rate at least $|x|^{2-n}$;
for $n=2$, $u$ tends to a quadratic polynomial plus $d\log |x|$ at infinity
with rate at least $|x|^{-1}$ for some constant $d$.
When $n=2$, L. Ferrer, A. Mart\'inez and F. Mil\'an obtained the same result
by adopting complex variable methods (cf. \cite{Ferrer-Martinez-Milan-1999-MathZ, Ferrer-Martinez-Milan-2000-MonatshMath}).

The main purpose of this paper is to extend the result in \cite{Caffarelli-Liyanyan-2003-CommPureApplMath} to the half space. Our main result is:
\begin{theorem}\label{TM_Main_TM}
Let $p(x)$ be a quadratic polynomial satisfying $D^2p>0$ and $f\geq0$ satisfy (\ref{SZ_supp_f-1_is_Bounded}).
Assume that u is a convex viscosity solution of (\ref{SZ_Main_TM_Equation}) such that
\begin{equation}\label{SZ_u_is_quadra_increasing_at_infty}
    \mu|x|^2\leq u(x)\leq {\mu}^{-1}|x|^2
     \quad\mbox{in } \mathbb{\overline{R}}^n_+\backslash B_{R_0}^+
\end{equation}
for some $0<\mu\leq\frac{1}{2}$.
Then there exist some symmetric positive definite matrix $A$ with $\det A=1$,
vector $b\in \mathbb{R}^n$ and constant $c\in\mathbb{R}$ such that
\begin{equation}\label{SZ_Main_TM_Asym_Behar_of_u}
   \left|u(x)-\left(\frac{1}{2}x^T Ax+b\cdot x+c\right)\right|\leq C\frac{x_n}{|x|^n}
   \quad \mbox{in }\mathbb{\overline{R}}^n_+\backslash B^+_R,
\end{equation}
where $x=(x',x_n)$, and $C$ and $R\geq R_0$ depend only on $n$, $\mu$ and $R_0$.
Moreover, $u\in C^{\infty} (\overline{\mathbb{R}}^n_+\backslash \Omega_0)$
and for any $k\geq1$,
\begin{equation}\label{SZ_Main_TM_Asym_Behar_of_D^ku}
     |x|^{n-1+k}\left|D^k\left(u(x)-\frac{1}{2}x^TAx-b\cdot x-c\right)\right|\leq C
     \quad \mbox{in }\mathbb{\overline{R}}^n_+\backslash B^+_R,
\end{equation}
where $C$ also depends on $k$.
\end{theorem}

\begin{remark}
(i) \eqref{SZ_u_is_quadra_increasing_at_infty} is reasonable. In fact,
the convex function
\[
    u(x_1,x_2)=\frac{x_1^2}{2(x_n+1)}+\frac{1}{2}(x_2^2+\cdots+x_{n-1}^2)+\frac{1}{6}(x_n^3+3x_n^2)
\]
solves
\begin{equation}\label{EQ_det=1_and_u=|x'|/2}
        \left\{
              \begin{aligned}
                    &\det D^2u=1\quad\mbox{in }\mathbb{R}^n_+,\\
                    &u=\frac{1}{2}|x'|^2\quad\quad\;\;\mbox{on }\{x_n=0\},\\
              \end{aligned}
       \right.
\end{equation}
 but it is not a quadratic polynomial (cf. \cite{Mooney-The-Monge-Ampere-equation,Savin-2014-JDifferentialEquations}).

(ii) Since the boundary condition, the asymptotic rate of solutions at infinity in exterior domains of the half space are faster than in exterior domains of the whole space, and also, the logarithm terms are ruled out as $n=2$.
$\hfill\Box$
\end{remark}

The following corollary is well known (cf.
\cite{Mooney-The-Monge-Ampere-equation,Savin-2014-JDifferentialEquations}) and is a simple consequence of Theorem \ref{TM_Main_TM}.

\begin{corollary}\label{Co_f=1}
    Let $u$ be a convex viscosity solution of
\begin{equation}\label{EQ_det=1_bdry=|x'|^2/2}
        \left\{
             \begin{aligned}
                  &\det D^2u=1\quad\mbox{in }\mathbb{R}^n_+,\\
                  &u=p(x)\quad\quad\mbox{ on }\{x_n=0\}\\
            \end{aligned}
       \right.
\end{equation}
 and satisfy (\ref{SZ_u_is_quadra_increasing_at_infty}),
 where $p(x)$ is a quadratic polynomial satisfying $D^2p>0$.
 Then $u$ is a quadratic polynomial.
\end{corollary}

The next theorem gives the existence and uniqueness of solutions of \eqref{SZ_Main_TM_Equation}
with prescribed asymptotic behavior at infinity.

\begin{theorem}\label{TM_Main_TM_extence}
Let $p(x)$ be a quadratic polynomial satisfying $D^2p>0$ and
$f\geq0$ be a bounded function satisfying (\ref{SZ_supp_f-1_is_Bounded}).
Then for any symmetric positive definite matrix $A$ with $\det A=1$,
vector $b\in \mathbb{R}^n$ and constant $c\in\mathbb{R}$ with the compatibility condition
$$p(x',0)=\frac{1}{2}(x',0)^TA(x',0)+b\cdot (x',0)+c,$$
there exists a unique convex solution $u\in C^{\infty} (\overline{\mathbb{R}}^n_+\backslash \Omega_0)$ of (\ref{SZ_Main_TM_Equation})
 satisfying
 \begin{equation}\label{SZ-ext-tend}
      \lim_{|x|\rightarrow\infty}\left|u(x)-\left(\frac{1}{2}x^TAx+b\cdot x+c\right)\right|=0.
 \end{equation}
 Moreover, \eqref{SZ_Main_TM_Asym_Behar_of_u} and \eqref{SZ_Main_TM_Asym_Behar_of_D^ku} hold.
\end{theorem}

The paper is organized as follows.
In Section 2, we introduce Pogorelov estimate in half domain and then give as a corollary the estimate of derivatives of solutions of Monge-Amp\`ere equations in half spaces.
In Section 3, it is obtained asymptotic behavior at infinity of solutions of
linear uniformly elliptic equations in exterior domains in half spaces that extends the results in D. Gilbarg and J. Serrin \cite{Gilbarg-Serrin-1995-JAnalyseMath} to the half space case.
In Section 4, to show Theorem \ref{TM_Main_TM},
 the idea in \cite{Caffarelli-Liyanyan-2003-CommPureApplMath} is borrowed and the results of Section 2 and Section 3 are applied to linearized equation of (\ref{SZ_Main_TM_Equation}).
In Section 5, Corollary \ref{Co_f=1} and Theorem \ref{TM_Main_TM_extence} are proved.

\bigskip

Throughout this paper, we use the following standard notations.
\begin{itemize}
\item For any $x\in \mathbb {R}^n$,  $x=(x_1,x_2,\cdots,x_n)=(x',x_n)$, $x'\in\mathbb {R}^{n-1}.$

\item $\mathbb {R}^n_+=\{x\in \mathbb {R}^n:x_n>0\}$; $\overline{\mathbb {R}}^n_+=\{x\in \mathbb {R}^n:x_n\geq0\}$.

\item For any $x\in \mathbb {R}^n$ and $r>0$, $B_r(x)=\{y\in \mathbb {R}^n:|y-x|<r\}$ and $B_r^+(0)=B_r(0)\cap\{x_n>0\}$. $B_r=B_r(0)$ and $B_r^+=B_r^+(0)$.

\item For any $r>0$, $Q_r^+=\{x\in \mathbb {R}^n:|x'|<r,~ 0<x_n<r\}$.

\end{itemize}

\section{Pogorelov estimate in half domain}

We start with the definition of viscosity solution.

\begin{definition}
Let $\Omega$ be an open subset of $\mathbb{R}^n$,
$u\in C(\Omega) $ be a convex function and $f\in C(\Omega)$, $f\geq0$.
The convex function $u$ is a viscosity \emph{subsolution} (\emph{supersolution}) of the equation
$\det D^2u=f$ in $\Omega$
if whenever convex $\phi(x)\in C^2(\Omega)$ and $x_0\in \Omega$ are such that
$(u-\phi)(x)\leq(\geq)(u-\phi)(x_0)$ for all $x$ in a neighborhood of $x_0$,
then we must have
\[
    \det D^2\phi(x_0)\geq(\leq)f(x_0).
\]

If $u$ is a viscosity subsolution and supersolution,
we call it viscosity solution.
$\hfill\Box$
\end{definition}

The following so called {\it Pogorelov estimate in half domain} was obtained by O. Savin \cite[Proposition 6.1, Remark 6.3 and Theorem 6.4]{Savin-2013-JAmerMathSoc}, which gives the boundary pointwise $C^{2,\alpha}$ estimates when the domain is not strictly convex.
It is also crucial to establish our main results.

\begin{theorem}[Pogorelov estimate in half domain]\label{TM_Pogorelov_estimates_in_half_domain}
Let $0<\rho_1$, $\rho_2$, $\rho_3<1$ be three constants and $\Omega$ be a convex domain such that \begin{equation*}
    B_{\rho_1}^+\subset \Omega \subset B_{\rho_1^{-1}}^+.
\end{equation*}
Assume that $u\in C(\overline\Omega)$ is a convex viscosity solution of
\begin{equation*}
     \left\{
           \begin{aligned}
                 &\det D^2u=1 \quad\;\; \mbox{in } \Omega,\\
                &u= p(x) \quad\quad\;\;\;\, \mbox{on } \{x_n=0\}\cap\partial\Omega,\\
                 &\rho_2\leq u\leq \rho_2^{-1} \quad\mbox{on }\{x_n>0\}\cap\partial\Omega,
           \end{aligned}
     \right.
\end{equation*}
where $p(x)$ is a quadratic polynomial satisfying
\begin{equation*}
    \rho_3|x'|^2\leq p(x',0)\leq\rho_3^{-1}|x'|^2.
\end{equation*}
Then there exists $c_0>0$ depending only on $n$, $\rho_1$, $\rho_2$ and $\rho_3$ such that \begin{equation}\label{SZ_PogoES}
    ||u||_{C^{3,1}(\overline{B}_{c_0}^+)}\leq c_0^{-1}.
\end{equation}
\end{theorem}

Theorem \ref{TM_Pogorelov_estimates_in_half_domain} together with interior Pogorelov estimates and the Schauder estimates implies the following corollary, which will be used to obtain the linear part of the quadratic polynomial in (\ref{SZ_Main_TM_Asym_Behar_of_u}).

\begin{corollary}\label{LM_D^kv_est}
Assume that $u\in C^\infty(\mathbb{\overline{R}}^n_+\backslash  B_1^+)$ satisfies
\begin{equation}\label{SZ-w}
     \left\{
           \begin{aligned}
  &[I_n+D^2u]>0,\quad\det (I_n+D^2u)=1\quad\mbox{and}\quad
   \left|u(x)\right|\leq \frac{\beta}{|x|^\gamma}
    \quad\mbox{in}~\mathbb{R}^n_+\backslash B^+_1,\\
    &u(x)=0\quad\mbox{on}~\{x: |x'|\geq1,\;x_n=0\}
    \end{aligned}
     \right.
  \end{equation}
for some constants $\beta>0$ and $\gamma>-2$.
Then
\[
    \left|D^ku(x)\right| \leq\frac{C}{|x|^{\gamma+k}}
    \quad \mbox{in }\mathbb{\overline{R}}^n_+\backslash B^+_{R_0}
\]
for any $ k\geq1$, where $R_0\geq1$ depends only on $n$, $\beta$ and $\gamma$, and $C$ depends only on $n$,  $\beta$, $\gamma$ and $k$.
\end{corollary}

\begin{proof}
We divide the proof into two cases.

{\it Case 1. Boundary point.}~
For any $x_0\in \{|x|=R\geq 3,x_n=0\}$,
let
\[
   \eta(x)=u(x)+\frac{1}{2}|x-x_0|^2,\quad x\in \mathbb{R}^n_+
\]
and
\[
   \eta_R(y)=\left(\frac{4}{R}\right)^2 \eta\left(x_0+\frac{R}{4}y\right),
   \quad y\in B^+_2.
\]
From (\ref{SZ-w}), it is easy to see that
\begin{equation*}
     \left\{
           \begin{aligned}
     &[D^2\eta_R(y)]> 0\quad\mbox{and}\quad\det D^2\eta_R(y)=1\quad\mbox{in}~B_2^+,\\
    &\eta_R(y)=\frac{1}{2}|y|^2\quad\mbox{on}~\partial B_2^+\cap\{y_n=0\},\\
    &1\leq \eta_R(y)\leq3\quad\mbox{on}~\{x_n>0\}\cap\partial B_2^+
 \end{aligned}
     \right.
  \end{equation*}
for $R\geq R_0$ as $R_0$ is large enough depending only on $\beta$ and $\gamma$.

By Theorem \ref{TM_Pogorelov_estimates_in_half_domain},
there exists $c_0>0$ depending only on $n$ such that
\begin{equation*}
  || \eta_R(y)||_{C^{3,1}(\overline{B}_{c_0}^+)}\leq c_0^{-1}.
\end{equation*}
This together with $\det D^2 \eta_R(y)=1$ implies
\begin{equation}\label{SZ-est-D2etaR}
     C^{-1}I_n\leq [D^2 \eta_R(y)]\leq CI_n
     \quad\mbox{ in }y\in \overline{B}_{c_0}^+,
\end{equation}
where $C$ depends only on $n$.
Differentiating $\ln(\det D^2\eta_R)=0$ with respect to $y_k$,
 we get
\[
     a_{ij}(y)D_{ij}({\eta_R})_k(y)=0
     \quad\mbox{in } B_{c_0}^+,
\]
where $ a_{ij}(y)=\{[D^2\eta_R]^{-1}\}_{ij}(y)$.
By the Schauder estimates,
\begin{equation*}\label{Sz-est-eta}
  ||\eta_R(y)||_{C^{k}(\overline{B}_{c_0/2}^+)}\leq C
\end{equation*}
for any $k\geq 1$, where $C$ depends only on $n$ and $k$.

Let
\[
    u_R(y)=\eta_R(y)-\frac{1}{2}|y|^2=\left(\frac{4}{R}\right)^2 u\left(x_0+\frac{ R}{4}y\right),
    \quad y\in B^+_2.
\]
By $\ln \det (D^2u_R+I_n)-\ln \det I_n=0$, we deduce
\begin{equation*}
    \left\{
        \begin{aligned}
            &\widetilde{a}_{ij}(y)D_{ij}u_R(y)=0
            \quad\mbox{in } B_{c_0/2}^+,\\
            &u_R(y)=0\quad\quad\quad\;\;\,\quad\mbox{on }\partial B_{c_0/2}^+\cap\{y_n=0\},\\
        \end{aligned}
    \right.
\end{equation*}
where $\widetilde{a}_{ij}(y)=\int_0^1[sD^2u_R(y)+I_n]^{ij}ds=
\int_0^1[sD^2\eta_R(y)+(1-s)I_n]^{ij}ds$ and
\[
    ||\widetilde{a}_{ij}||_{C^{k}(\overline{B}_{c_0/2}^+)}\leq C
       \quad\mbox{and}\quad
    C^{-1}I_n\leq [\widetilde{a}_{ij}]\leq CI_n
      \quad\mbox{on } \overline{B}_{c_0/2}^+.
\]
for any $k\geq 1$.
By the Schauder estimates,
\begin{align*}
    ||D^ku_R(y)||_{L^\infty(\overline{B}_{{c_0}/{4}}^+)}\leq C||u_R(y)||_{L^\infty(\overline{B}_{{c_0}/{2}}^+)}
             \leq CR^{-\gamma-2}
\end{align*}
for any $k\geq 1$,
where $C$ depends only on $n$,  $\beta$, $\gamma$ and $k$.

It follows that
\begin{equation}\label{SZ-Dkw-esti}
  |D^ku(x)|\leq \frac{C}{|x|^{\gamma+k}}
   \quad\mbox{in } \overline{B}_{\theta R}^+(x_0),
\end{equation}
for any $ k\geq1$,
where $\theta=\frac{c_0}4$ and $R_0\geq 1$
depend only on $n$, $\beta$ and $\gamma$,
and $C$ depends only on $n$, $\beta$, $\gamma$ and $k$.

{\it Case 2. Interior point.}
For any $x_0\in \{|x|=R\geq R_0,x_n\geq\theta R\}$, by similar arguments as {\it Case 1}, where we use interior Pogorelov estimates instead of
Theorem \ref{TM_Pogorelov_estimates_in_half_domain} (cf. \cite[Lemma 3.5]{Caffarelli-Liyanyan-2003-CommPureApplMath}), we obtain
\[
    |D^ku(x_0)|\leq\frac{C}{|x_0|^{\gamma+k}}
\]
for any $ k\geq1$.
\end{proof}

\section{Asymptotic behavior of linear elliptic equation in half space}

In this section we study the asymptotic behavior at infinity
of solutions of linear uniformly elliptic equations in non-divergence form outside a bounded domain of $\mathbb{R}^n_+$, which extends the results in D. Gilbarg and J. Serrin \cite{Gilbarg-Serrin-1995-JAnalyseMath} to half spaces and will be applied to the linearized equation of
(\ref{SZ_Main_TM_Equation}) in the proof of Theorem \ref{TM_Main_TM}.

We begin with the following two auxiliary lemmas.

\begin{lemma}\label{Lm_Int_Str_Small}
Let $R_0>0$ and $a_{ij}(x)\in C(\overline{B}_{4R_0}^+\backslash B_{R_0}^+)$ such that $\lambda I\leq [a_{ij}(x)]\leq \Lambda I$
in $\overline{B}_{4{R_0}}^+\backslash {B}_{R_0}^+$
for some $0<\lambda\leq\Lambda<\infty$.
Assume that
\begin{equation}\label{EQ_supsolu}
     \left\{
        \begin{aligned}
           &a_{ij}(x)D_{ij}u(x)=0\quad\mbox{in } B_{4{R_0}}^+\backslash \overline{B}_{R_0}^+,\\
           &u(x)\leq1 \quad\quad\quad\quad\;\;\mbox{ on }\partial( B_{4{R_0}}^+\backslash \overline{B}_{R_0}^+)\cap\{x_n>0\},\\
           &u(x) \leq\frac{1}{2}\quad\quad\;\;\quad\quad\mbox{on }
              \partial(B_{4{R_0}}^+\backslash\overline{B}_{R_0}^+)\cap\{x_n=0\}.
        \end{aligned}
     \right.
\end{equation}
Then
\[
    u(x)\leq 1-\varepsilon_0\quad \mbox{on } \partial{B}_{2{R_0}}\cap\{x_n\geq0\}
\]
for some $\varepsilon_0>0$
depending only on $n$, $\lambda$ and $\Lambda$.
\end{lemma}

\begin{proof}
Without loss of generality, it is assumed that ${R_0}=1$.

By the third inequality in \eqref{EQ_supsolu}
and the H$\ddot{\mbox{o}}$lder continuity up to the boundary of $u$,
there exists constant $0<\delta\leq1$ depending only on
 $n$, $\lambda$ and $\Lambda$  such that
\begin{equation}\label{SZ-u-lager-1}
    u(x)\leq \frac{2}{3}
    \quad \mbox{on }\partial B_2\cap\{0\leq x_n\leq\delta\}.
\end{equation}
Applying Harnack inequality to $1-u$,
there exists  a positive constant $C\geq1$
depending only on  $n$, $\lambda$ and $\Lambda$ such that
\[
    C\inf\limits_{\partial B_2\cap\{x_n\geq\delta\}} (1-u)
    \geq\sup\limits_{\partial B_2\cap\{x_n\geq\delta\}} (1-u)
    \geq\sup\limits_{\partial B_2\cap\{x_n=\delta\}} (1-u)
    \geq \frac{1}{3}.
\]
It implies that
\begin{equation}\label{SZ-u-lager-2}
   u(x)\leq 1-\frac{1}{3C}   \quad \mbox{on }\partial B_2\cap\{x_n\geq\delta\}.
\end{equation}

This lemma follows immediately from \eqref{SZ-u-lager-1} and \eqref{SZ-u-lager-2} by taking $\varepsilon_0=\frac{1}{3C}$.
\end{proof}

\begin{lemma}\label{Lm_Limits_of_solutions_of_linear_Eq}
Let $R_0>0$ and $a_{ij}(x)\in C(\overline{\mathbb{R}}^n_+\backslash B_{R_0}^+)$ such that $\lambda I\leq [a_{ij}(x)]\leq \Lambda I$  in $\overline{\mathbb{R}}^n_+\backslash B_{R_0}^+$ for some $0<\lambda\leq\Lambda<\infty$. Assume that
\begin{equation*}
     \left\{
        \begin{aligned}
   & a_{ij}(x)D_{ij}u(x)=0\quad\mbox{in }\mathbb{R}^n_+\backslash B_{R_0}^+,\\
   & |u(x)|\leq1\quad\quad\quad\quad\;\mbox{on }(\partial B_{R_0}\cap\{x_n>0\})\cup\{x_n=0,|x|\geq R_0\},\\
   & u(x',0)\rightarrow 0\quad\quad\;\,\quad\mbox{as }|x'|\rightarrow\infty,\\
 &|Du(x)|\rightarrow 0\quad\quad\quad\,\mbox{as }|x|\rightarrow \infty.
        \end{aligned}
    \right.
\end{equation*}
Then $|u|\leq1$ in $\mathbb{\overline{R}}^n_+\backslash B_{R_0}^+$.
\end{lemma}

\begin{proof}
For any $\varepsilon>0$, since $|Du|\rightarrow 0$ as $|x|\rightarrow\infty$,
there exists $R_\varepsilon\geq R_0$ such that
\begin{equation}\label{SZ-Du-small}
  |Du|\leq \varepsilon
  \quad \mbox{in } \mathbb{\overline{R}}^n_+\backslash Q_{R_\varepsilon}^+,
\end{equation}
where
$Q_{R_\varepsilon}^+=\{(x',x_n):|x'|<R_\varepsilon, 0<x_n<R_\varepsilon\}$.

Since $|u|\leq 1$ on $\{x_n=0,~|x|\geq R_0\}$, we have,
by \eqref{SZ-Du-small},
\[
    |u(x)|\leq 1+2\varepsilon x_n
    \quad \mbox{on } \partial Q_{R_\varepsilon}^+\cap\{x_n>0\}.
\]
In view of $|u|\leq1$
on $(\partial B_{R_0}\cap\{x_n>0\})\cup\{x_n=0,|x|\geq R_0\}$, we deduce
by the comparison principle that
\[
    |u(x)|\leq 1+2\varepsilon x_n
    \quad\mbox{in } \overline{Q}_{R_\varepsilon}^+\backslash {B}_{R_0}^+.
\]
Fix any $x\in \mathbb{\overline{R}}^n_+\backslash B_{R_0}^+$ and we derive the conclusion by letting $\varepsilon\rightarrow 0$.
\end{proof}

The following two theorems are our main results of this section and will be used in the proof of Theorem 1.1.

\begin{theorem}\label{Co_Limits_of_solutions_of_linear_Eq}
Let $R_0>0$ and $a_{ij}(x)\in C(\overline{\mathbb{R}}^n_+\backslash B_{R_0}^+)$ such that $\lambda I\leq [a_{ij}(x)]\leq \Lambda I$
 in $\overline{\mathbb{R}}^n_+\backslash B_{R_0}^+$
 for some $0<\lambda\leq\Lambda<\infty$.
Assume that
\begin{equation*}
     \left\{
        \begin{aligned}
   & a_{ij}(x)D_{ij}u(x)=0
    \quad\mbox{in }\mathbb{R}^n_+\backslash B_{R_0}^+,\\
   & |u|\leq1\quad\quad\quad\quad\quad\;\,\,\mbox{on }(\partial B_{R_0}\cap\{x_n>0\})\cup\{x_n=0,|x|\geq R_0\},\\
 & u(x',0)\rightarrow \beta\quad\quad\quad\,\,\mbox{as }|x'|\rightarrow\infty,\\
 & |Du(x)|\rightarrow 0\quad\quad\quad\,\mbox{as }|x|\rightarrow \infty
         \end{aligned}
    \right.
\end{equation*}
for some real number $\beta$. Then $u(x)\rightarrow \beta$ as $|x|\rightarrow\infty$.
\end{theorem}

\begin{proof}
Suppose that $\beta=0$. Otherwise, we consider
\[
    ({u(x)-\beta})/({1+|\beta|}).
\]

By Lemma \ref{Lm_Limits_of_solutions_of_linear_Eq},
we have $|u(x)|\leq 1$ in $\mathbb{\overline{R}}^n_+\backslash B_{R_0}^+$.
Then $u$ has finite superior limit $\overline{u}$
and inferior limit $\underline{u}$ at infinity.
By $u(x',0)\rightarrow 0$ as $|x'|\rightarrow\infty$,
we get $\overline{u}\geq 0\geq \underline{u}$.

Now we argue by contradiction.
If this theorem is not true, then
$\overline{u}>0$ or $\underline{u}<0$.
We may assume $\overline{u}>0$.
Otherwise, we consider $-u$.

Let $\varepsilon_0$ be given by Lemma \ref{Lm_Int_Str_Small}.
By the definition of $\overline{u}$,
there exists large $R_1\geq R_0$ such that
for all $R\geq R_1$,
\[
    u(x)\leq(1+\frac{\varepsilon_0}{2})\overline{u}
    \quad\mbox{in }\mathbb{\overline{R}}^n_+\backslash B_R^+
\]
and (by $u(x',0)\rightarrow\beta=0~\mbox{as}~|x'|\rightarrow\infty$,)
\[
    u(x',0)\leq \frac{1}{2}(1+\frac{\varepsilon_0}{2})\overline{u}
    \quad\mbox{on }\{x_n=0,~|x'|>R\}.
\]

Applying Lemma \ref{Lm_Int_Str_Small}
to $\frac{u(x)}{(1+\varepsilon_0/2)\overline{u}}$
in $B_{4R}^+\backslash \overline{B}_{R}^+$,
we get
\[
    u(x)\leq(1-\varepsilon_0)(1+\frac{\varepsilon_0}{2})\overline{u}
         \leq (1-\frac{\varepsilon_0}{2})\overline{u}
           \quad  \mbox{on }\partial B_{2R}\cap\{x_n\geq0\}.
\]
However, by the arbitrariness of $R\geq R_1$,
\[
    u(x) \leq(1-\frac{\varepsilon_0}{2})\overline{u}
    \quad  \mbox{in}~\overline{\mathbb{R}}^n_+\backslash B_{2R_1}^+,
\]
which contradicts the definition of $\overline{u}$.
\end{proof}

\begin{theorem}\label{TM_Asymp_Behr_of_solu_of_Linear_Eq}

Let $R_0>0$ and $a_{ij}(x)\in C(\overline{\mathbb{R}}^n_+\backslash B_{R_0}^+)$ such that $\lambda I\leq [a_{ij}(x)]\leq \Lambda I$ and $|a_{ij}(x)- \delta_{ij}|\leq |x|^{-s}$
 in $\overline{\mathbb{R}}^n_+\backslash B_{R_0}^+$
 for some $s>0$ and $0<\lambda\leq\Lambda<\infty$.
 Assume that
\begin{equation*}
     \left\{
        \begin{aligned}
              &a_{ij}(x)D_{ij}u(x)=0\quad\mbox{in }\mathbb{R}^n_+\backslash B_{R_0}^+, \\
              &u(x)=0\quad\quad\quad\quad\;\;\,\mbox{on }\{x_n=0,|x|\geq R_0\},\\
              &|u|\leq1\quad\quad\quad\quad\;\,\quad\mbox{on }\partial B_{R_0}\cap\{x_n>0\},\\
&|Du(x)|\leq 1\quad\quad\quad\;\,\mbox{in }\mathbb{\overline{R}}^n_+\backslash B_{R_0}^+,\\
&|Du(x)|\rightarrow 0\quad\quad\quad\,\mbox{as }|x|\rightarrow \infty.
                        \end{aligned}
     \right.
\end{equation*}
Then we have
\begin{equation}\label{SZ_AsB_Lin}
    |u(x)|\leq C\frac{x_n}{|x|^n}
    \quad\mbox{in }\mathbb{\overline{R}}^n_+\backslash B_R^+,
\end{equation}
where $C$ and $R\geq R_0$
depend only on $n$, $s$ and $R_0$.
\end{theorem}

\begin{proof}
 By Theorem \ref{Co_Limits_of_solutions_of_linear_Eq},
\[
    u(x)\rightarrow0\quad \mbox{as } |x|\rightarrow\infty.
\]

Fix $0<\delta <\min\{1,\frac{s}{n-1}\}$ and let
\[
    {w}(x)=\frac{x_n}{|x|^n}-\left(\frac{x_n}{|x|^n}\right)^{1+\delta}.
\]
Then for all $1\leq i$, $j\leq n$,
\[
    D_iw=\left(1-(1+\delta)\left(\frac{x_n}{|x|^n}\right)^\delta\right)
    \left(\frac{\delta_n^i}{|x|^n}-\frac{nx_nx_i}{|x|^{n+2}}\right)
\]
and
\begin{align*}
  D_{ij}w=
   &\left(1-(1+\delta)\left(\frac{x_n}{|x|^n}\right)^\delta\right)
        \left(\frac{-n(\delta_n^ix_j+\delta_n^jx_i+\delta_j^ix_n)}{|x|^{n+2}}
         +\frac{n(n+2)x_nx_ix_j}{|x|^{n+4}}\right)\\
   &+\left(-\delta(1+\delta)\left(\frac{x_n}{|x|^n}\right)^{\delta-1}\right)
        \left(\frac{\delta_n^i}{|x|^n}-\frac{nx_nx_i}{|x|^{n+2}}\right)
        \left(\frac{\delta_n^j}{|x|^n}-\frac{nx_nx_j}{|x|^{n+2}}\right).
\end{align*}
Consequently
\begin{equation}\label{SZ-Lap-w}
  \begin{split}
  \Delta w=&-\delta(1+\delta)\left(\frac{x_n}{|x|^n}\right)^{\delta-1}
             \left(\frac{1}{|x|^{2n}}+\frac{(n^2-2n) x_n^2}{|x|^{2n+2}}\right)\\
          \leq &-\delta(1+\delta)\frac{1}{|x|^{2n}}\left(\frac{x_n}{|x|^n}\right)^{\delta-1}
  \end{split}
\end{equation}
and there exists $C$ depending only on $n$ and $\delta$
such that for all $1\leq i$, $j\leq n$,
\begin{equation}\label{SZ-Dij-w}
  \begin{split}
  |D_{ij}w|&\leq C\left( \frac{1}{|x|^{n+1}}+ \frac{1}{|x|^{n+1}}\left(\frac{x_n}{|x|^n}\right)^{\delta}+
         \frac{1}{|x|^{2n}}\left(\frac{x_n}{|x|^n}\right)^{\delta-1}\right)\\
       &\leq C\left( \frac{1}{|x|^{n+1}}+\frac{1}{|x|^{2n}}\left(\frac{x_n}{|x|^n}\right)^{\delta-1}\right).
  \end{split}
\end{equation}

By \eqref{SZ-Lap-w}, \eqref{SZ-Dij-w}
and $|a_{ij}(x)- \delta_{ij}|\leq |x|^{-s},$
there exists $C$ depending only on $n$, $\delta$ and $s$
such that
\begin{equation*}
  \begin{split}
     a_{ij}(x)D_{ij}w(x)
    &= \delta_{ij} D_{ij}w(x)+(a_{ij}(x)- \delta_{ij})D_{ij}w(x)\\
    &\leq -\delta(1+\delta)\frac{1}{|x|^{2n}}\left(\frac{x_n}{|x|^n}\right)^{\delta-1}\\
    &\quad +C |x|^{-s}\left( \frac{1}{|x|^{n+1}}+\frac{1}{|x|^{2n}}
           \left(\frac{x_n}{|x|^n}\right)^{\delta-1}\right)\\
    &\leq\left(-\delta(1+\delta)+C |x|^{-s}\right)|x|^{-2n}\left(\frac{x_n}{|x|^n}\right)^{\delta-1}
           +C |x|^{-s-n-1}.
  \end{split}
\end{equation*}
It follows that there exists $R_1\geq 2R_0$ large enough
(depending only on $n$, $\delta$, $s$ and $R_0$) such that
\begin{equation}\label{SZ_W_supsolution}
    a_{ij}(x)D_{ij}w(x)\leq 0
    \quad\mbox{in } \mathbb{R}^n_+\backslash \overline{B}^+_{R_1},
\end{equation}
where
$
    0<\delta <\min\{1,\frac{s}{n-1}\}
$
and
$\left(\frac{x_n}{|x|^n}\right)^{\delta-1}\geq |x|^{-(n-1)(\delta-1)}$ are used.

From $|Du(x)|\leq 1$ in $\mathbb{R}^n_+\backslash B_{R_0}^+$ and $u(x)=0~\mbox{on}~\{x_n=0,|x|\geq R_0\}$,
we deduce
\[
|u(x)|\leq 3 x_n\quad \mbox{on }\partial{B}_{R_1}\cap\{x_n\geq0\}.
\]
On the other hand, on $\partial{B}_{R_1}\cap\{x_n\geq0\}$,
\begin{equation}\label{SZ-u-less-w-1}
 \begin{split}
   &w(x)=\frac{x_n}{|x|^n}\left(1-\left(\frac{x_n}{|x|^n}\right)^{\delta}\right)
   \geq\frac{x_n}{|x|^n}\left(1-\left(\frac{1}{|x|^{n-1}}\right)^{\delta}\right)\\[10pt]
  &\mbox{}\hskip0.5cm =\frac{x_n}{R_1^n}\left(1-R_1^{(1-n)\delta}\right)\geq C^{-1}|u(x)|,
 \end{split}
\end{equation}
where $C=\frac{3R_1^n}{1-R_1^{(1-n)\delta}}$.

For any $\epsilon>0$, by $u(x)\rightarrow 0$ as $|x|\rightarrow\infty$,
there exists $R_{\epsilon}>R$ such that
\begin{equation*}
   |u(x)|\leq \epsilon,\quad x\in\partial B_{R_{\epsilon}}\cap\{x_n\geq0\}.
\end{equation*}
Combine it with \eqref{SZ-u-less-w-1} and \eqref{SZ_W_supsolution} and
by the comparison principle, we have
\[
    |u(x)|\leq C{w}(x)+\epsilon
    \quad\mbox{in } \overline{B}^+_{R_{\epsilon}}\backslash B^+_{R}.
\]
Let $\epsilon \rightarrow 0$ and we conclude (\ref{SZ_AsB_Lin}).
\end{proof}

\section{Proof of Theorem \ref{TM_Main_TM}}

In this section we prove Theorem \ref{TM_Main_TM}, which is
equivalent to the following:

\begin{theorem}\label{CoT2infty}
Assume that $f\geq0$ satisfying
\begin{equation}\label{SZ_supp_f-1_is_Bounded_in_B1}
    \Omega_0=\mbox{support} (f-1)\subset B_1^+
\end{equation}
and u is a convex viscosity solution of
\begin{equation}\label{EQ_det=f_u=|x'|^2/2}
    \left\{
        \begin{aligned}
        &\det D^2u=f\quad\mbox{in } \mathbb{R}^n_+,\\
        &u=\frac{1}{2}|x|^2\quad\quad\mbox{on }  \{x_n=0\} \\
        \end{aligned}
    \right.
\end{equation}
satisfying
\begin{equation}\label{SZ_u_is_quadra_increasing_outside_B^+_1}
   \mu|x|^2\leq u\leq \mu^{-1}|x|^2
   \quad\mbox{in } \mathbb{\overline{R}}^n_+\backslash B_{1}^+
\end{equation}
for some $0<\mu\leq\frac12$.
Then there exist some symmetric positive definite matrix $A$
with $\det A=1$ and constant $b_n\in\mathbb{R}$
such that
\begin{equation}\label{SZ_TM_Asymp_behr_of_u}
    \left|u(x)-\left(\frac{1}{2}x^T Ax+b_nx_n\right)\right|
    \leq C\frac{x_n}{|x|^n}
   \quad\mbox{in } \mathbb{\overline{R}}^n_+\backslash B_{R}^+,
\end{equation}
where $C$ and $R$ depend only on $n$ and $\mu$.
Moreover, $u\in C^{\infty}(\overline{\mathbb{R}}^n_+\backslash{\Omega_0})$
and
\begin{equation}\label{SZ_TM_Asymp_behr_of_D^ku}
    |x|^{n-1+k}
     \left|D^k\left(u(x)-\frac{1}{2}x^T Ax-b_nx_n\right)\right|\leq C
   \quad\mbox{in } \mathbb{\overline{R}}^n_+\backslash B_{R}^+,
\end{equation}
where $C$ also depends on $k$.
\end{theorem}

\begin{remark}
Theorem \ref{TM_Main_TM} follows easily from Theorem \ref{CoT2infty}. Indeed,
after subtracting an affine function, we may assume that in (\ref{SZ_Main_TM_Equation}), $p(x)$ is homogeneous of degree 2, that is, $p(x)=\frac{1}{2}{x}^T P x$ for some $n\times n$ symmetric positive definite matrix $P=\begin{bmatrix}\widetilde{P}&\nu\\\nu^T&a\end{bmatrix}$, where $\widetilde{P}\in \mathbb{R}^{(n-1)\times(n-1)}$, $\nu\in\mathbb{R}^{n-1}$ and $a\in\mathbb{R}$. Let $\widetilde{Q}\in\mathbb{R}^{(n-1)\times(n-1)}$ such that $\widetilde{Q}^T\widetilde{P}\widetilde{Q}=I_{n-1}$ and $Q=\begin{bmatrix}\widetilde{Q}&0\\0&\frac1{\det\widetilde{Q}}\end{bmatrix}$.
It is easy to see that $\det Q=1$ and $p(Qx)\Big|_{x_n=0}=\frac12|x'|^2$. Then we can deduce Theorem 1.1 by Theorem 4.1.
$\hfill\Box$
\end{remark}

We first show the smoothness of $u$.

\begin{lemma}\label{Lm_u_is_smooth_outside}
    Let $u$ be given by Theorem \ref{CoT2infty}.
    Then $u\in C^{\infty}(\overline{\mathbb{R}}^n_+\backslash{\Omega_0})$.
\end{lemma}

\begin{proof}
For any $x_0\in \{x_n=0\}\backslash{\Omega_0}$,
let $d=dist(x_0,\Omega_0)$ and
 \[
     \widetilde{u}(x)=u(x)-u(x_0)-x_0\cdot(x-x_0)+Cx_n,
 \]
where $C$ is large such that $\widetilde{u}$ is positive
on $\partial B_d(x_0)\cap\{x_n\geq0\}$.
By Theorem \ref{TM_Pogorelov_estimates_in_half_domain} and the Schauder estimates, we get
$u\in C^\infty(\overline{B}^+_{c}(x_0))$ for some positive $c<d$.

For any $x_0\in \{x_n>0\}\backslash{\Omega_0}$,
let $\widetilde{d}=dist(x_0,\partial(\mathbb{R}^n_+\backslash{\Omega_0}))$
and $l_{x_0}(x)$ be a support plane of $u$ at $x_0$.
Then
     $\alpha:=\min\limits_{\partial B_{\widetilde{d}/2}(x_0)}(u-l_{x_0})(x)>0$.
Indeed, if $\alpha=0$, then
by \cite[Theorem 1]{Caffarelli-1990-AnnofMath-Aloca},
there exists an endless line $L\subset\{u-l_{x_0}=0\}$ and this
contradicts \eqref{SZ_u_is_quadra_increasing_outside_B^+_1}.
Therefore $\{x:u<l_{x_0}+\alpha\}\subset B_{\widetilde{d}/2}(x_0)$ and then
we get $u\in C^{\infty}$ in $\{x:u<l_{x_0}+\alpha\}$ by \cite{Caffarelli-Liyanyan-2003-CommPureApplMath, Cheng-Yau-1977-Comm.PureAppl.Math.}.
\end{proof}

In the following, the remaining part of the proof of Theorem \ref{CoT2infty} includes two stages:
nonlinear approach and linear approach.
In the nonlinear approach, we follow the idea of \cite{Caffarelli-Liyanyan-2003-CommPureApplMath} and show that there exist some matrix $T$ and constant $\epsilon>0$ such that $|u(Tx)-\frac{1}{2}|x|^2|= O(|x|^{2-\epsilon})$ at infinity by approximating $u$ with a good function $\xi$ which can be estimated by Pogorelov estimate in half domain.
In the linear approach, we obtain some linear function $l(x)$ such that
$|u(Tx)-\frac{1}{2}|x|^2-l(x)|=O\left(\frac{x_n}{|x|^n}\right)$ at infinity by linearizing the equation and using the results we established in Section 2 and Section 3.

\subsection{Nonlinear approach}

Denote the cross section by
\[
 S_M(u)=\{x\in\mathbb{\overline{R}}^n_+:u(x)<M\}.
\]

By \eqref{SZ_u_is_quadra_increasing_outside_B^+_1},
for any $M\geq\mu^{-1}$,
\begin{equation*}
    M^{1/2}\mu^{1/2}\overline{B}^+_1\subset
               S_M(u)
    \subset M^{1/2}\mu^{-1/2}\overline{B}^+_1.
\end{equation*}

Let
\begin{equation*}
  \widehat{u}(x)=\frac{1}{M}u(M^{1/2}x),
  \quad x\in \mathcal{O}:=\frac{1}{M^{1/2}}S_M(u).
\end{equation*}
Then
\begin{equation*}
   \mu^{1/2}\overline{B}_1^{+}\subset
             \mathcal{O}
   \subset \mu^{-1/2}\overline{B}_1^{+}
\end{equation*}
and
\begin{equation*}
\left\{
\begin{aligned}
    &\det D^2\widehat{u}=f(M^{1/2}x)
    \quad\mbox{in } \mathcal{O},\\
    &\widehat{u}=\frac{1}{2}|x|^2
       \quad\quad\quad\quad\quad\;\;\,
       \mbox{on }\partial\mathcal{O}\cap\{x_n=0\},\\
    &\widehat{u}=1
       \quad\quad\quad\quad\quad\quad\quad\;\,\,
       \mbox{on }\partial\mathcal{O}\cap\{x_n>0\}.
\end{aligned}
\right.
\end{equation*}

Now we consider the following Dirichlet problem
\begin{equation*}
  \left\{
    \begin{aligned}
      &\det D^2\xi=1
      \quad\mbox{in }\mathcal{O},\\
      &\xi=\frac{1}{2}|x|^2
      \quad\;\;\;\,\mbox{on } \partial\mathcal{O}\cap\{x_n=0\},\\
      &\xi=1\quad
      \quad\quad\;\;\;\mbox{on } \partial\mathcal{O}\cap\{x_n>0\}.
    \end{aligned}
  \right.
\end{equation*}
We refer to \cite{Aleksandrov-1958-VestnikLeningrad, Bakel'man-1957-DoklAkadNauk, Cheng-Yau-1977-Comm.PureAppl.Math.} for the existence of $\xi$. Observe
 the boundary value of $\xi$ can be extended  to a convex function on $\overline{\mathcal{O}}$ by defining
 \[
    \sup\big\{l(x):~l ~\mbox{is an affine function and}~ l\leq \xi|_{\partial\mathcal{O}} \mbox{ on }\partial\mathcal{O}\big\},
    \quad \forall x\in \mathcal{O}.
\]

For any $M\geq\mu^{-1}$,
applying Theorem \ref{TM_Pogorelov_estimates_in_half_domain}
to $\xi$ in $\mathcal{O}$,
there exists $c_0>0$ depending only on $n$ and $\mu$
such that
\begin{equation}\label{SZ_D^kXi_bounded}
    |D\xi(x)|\leq c_0^{-1}, \quad c_0 I\leq [D^2\xi(x)]\leq c_0^{-1} I\quad\mbox{and}\quad
    |D^3\xi(x)|\leq c_0^{-1}\quad \mbox{in } \overline{B}_{c_0}^+.
\end{equation}

\begin{lemma}\label{Lm_u^hat-Xi_small}
For any $M\geq\max\{\mu^{-1}, c_0^{-2}\}$,
\begin{equation*}
      |\widehat{u}-\xi|\leq CM^{-1/2}
      \quad \mbox{ in }\mathcal{\overline{O}}\backslash {B}_{M^{-1/2}}^+,
\end{equation*}
where $C$ depends only on $n$ and $\mu$.
\end{lemma}

\begin{proof}By (\ref{SZ_u_is_quadra_increasing_outside_B^+_1}), we see that
\begin{equation*}
    \mu M^{-1}\leq \widehat{u}\leq \mu^{-1} M^{-1}
    \quad\mbox{on } \partial{B}_{M^{-1/2}}^+\cap\{x_n>0\}.
\end{equation*}
In view of (\ref{SZ_D^kXi_bounded}),
\begin{equation*}
     |\xi(x)|\leq c_0^{-1}M^{-1/2}
     \quad\mbox{in }\overline{B}_{M^{-1/2}}^+
\end{equation*}
for any $M\geq\max\{\mu^{-1}, c_0^{-2}\}$. Therefore
\begin{equation*}
      |\widehat{u}-\xi|\leq CM^{-1/2}
      \quad\mbox{on } \partial{B}_{M^{-1/2}}^+\cap\{x_n>0\},
\end{equation*}
where $C$ is a constant depending only on $n$ and $\mu$.

Observe $f\equiv1$ in $\mathcal{\overline{O}}\backslash {B}_{M^{-1/2}}^+$. Then the comparison principle follows this lemma.
\end{proof}

\begin{lemma}\label{LM_Dxi(y^ba)=0}
For any $M\geq\max\{\mu^{-2}, c_0^{-2}\}$,
\begin{equation}\label{EQ1}
|D\xi(0)|\leq CM^{-1/4},
\end{equation}
where $C$ depends only on $n$ and $\mu$.
\end{lemma}

\begin{proof}
Since $\xi=\frac{1}{2}|x|^2$ on $\{x_n=0\}$,
it suffices to show $|D_n\xi(0)|\leq CM^{-1/4}$.

For any $\widetilde{x}=(0,\widetilde{x}_n)\in \overline{B}_{c_0}^+$, by (\ref{SZ_D^kXi_bounded}), there exists $\theta\in[0,1]$ such that
\[
     \xi(0,\widetilde{x}_n)=\xi(0)+D_n\xi(0)\cdot \widetilde{x}_n+\frac{1}{2} D_{nn}\xi(\theta\widetilde{x})\widetilde{x}_n^2,
\]
which gives
\begin{equation*}
   |D_n\xi(0)|
    \leq \frac{|\xi(0,\widetilde{x}_n)|+
    \frac{1}{2} D_{nn}\xi(\theta\widetilde{x})\widetilde{x}_n^2}{\widetilde{x}_n}
    \leq \frac{|\xi(0,\widetilde{x}_n)|+
   C\widetilde{x}_n^2}{\widetilde{x}_n}
\end{equation*}
for some constant $C$ depending only on $n$ and $\mu$.

Choose $\widetilde{x}_n$ such that $\widehat{u}(0,\widetilde{x}_n)= M^{-1/2}$.
Then, by \eqref{SZ_u_is_quadra_increasing_outside_B^+_1},
\begin{equation*}
   M^{-1/4}\mu^{1/2}\leq\widetilde{x}_n\leq M^{-1/4}\mu^{-1/2}
\end{equation*}
for any $M\geq\max\{\mu^{-2}, c_0^{-2}\}$. Since, by Lemma \ref{Lm_u^hat-Xi_small},
\begin{equation*}
  |\xi(0,\widetilde{x}_n)|\leq|\widehat{u}(0,\widetilde{x}_n)|+CM^{-1/2}
  \leq CM^{-1/2}
\end{equation*}
for some constant $C$ depending only on $n$ and $\mu$, we see (\ref{EQ1}) clearly.
\end{proof}

Let
\begin{equation}\label{SZ_DF_EM}
      E_M=\Big\{x\in\mathbb{\overline{R}}^n_+:\;x^T D^2\xi(0)x\leq 1\Big\}.
\end{equation}

\begin{lemma}\label{Lm_Level_set_A_M}Let $\tau=\frac{1}{10}$.
There exist $k_0$ and $\widetilde{C}$, depending only on $n$ and $\mu$, such that for all $ k\geq k_0$,
$M=2^{(1+\tau)k}$ and $M'\in [2^{k-1}, 2^k]$,
\begin{equation}\label{SZ_SECT_CONTAINING}
    \left(\frac{2 M'}{M}-\widetilde{C}2^{-\frac{3}{2}\tau k}\right)^{1/2}E_{M}
          \subset \frac{ S_{M'}(u)}{M^{1/2}}\subset
    \left(\frac{2 M'}{M}+\widetilde{C}2^{-\frac{3}{2}\tau k}\right)^{1/2}E_{M}.
\end{equation}
\end{lemma}

\begin{proof}
By Lemma \ref{Lm_u^hat-Xi_small} and
\[
  \frac{S_{M'}(u)}{M^{1/2}}=\left\{\widehat{u}<\frac{M'}{M}\right\},
\]
we have
\begin{equation*}
   \left\{\xi<\frac{M'}{M}-\frac{C}{M^{1/2}}\right\}
          \subset \frac{S_{M'}(u)}{M^{1/2}}\subset
   \left\{\xi<\frac{M'}{M}+\frac{C}{M^{1/2}}\right\},
\end{equation*}
where
$B^+_{M^{-1/2}}\subset\left\{\widehat{u}<\frac{M'}{M}\right\}$ and $B^+_{M^{-1/2}}\subset\left\{\xi<\frac{M'}{M}-\frac{C}{M^{1/2}}\right\}$ are used.

For any $x\in \overline{B}_{c_0}^+$, by (\ref{SZ_D^kXi_bounded}),
\begin{equation*}
     \left|\xi(x)-\xi(0)-D\xi(0)\cdot x-\frac{1}{2}x^T D^2\xi(0)x\right|
     \leq c_0^{-1}|x|^3.
\end{equation*}
From $\xi(0)=0$ and Lemma \ref{LM_Dxi(y^ba)=0}, it follows that
\begin{equation*}
\begin{aligned}
  &\frac{1}{2}x^T D^2\xi(0)x-CM^{-1/4}|x|-c_0^{-1}|x|^3\\  &
  \mbox{}\hskip2cm\leq\xi(x)\leq \frac{1}{2}x^T D^2\xi(0)x+CM^{-1/4}|x|+c_0^{-1}|x|^3.
\end{aligned}
\end{equation*}
Since, by
(\ref{SZ_D^kXi_bounded}), $c_0 I\leq [D^2\xi(0)]\leq c_0^{-1}$,
we have $x\in\left\{\xi<\frac{M'}{M}-\frac{C}{M^{1/2}}\right\}$ implies
$|x|\leq C\left(\frac{M'}{M}\right)^{1/2}$ for some constant $C$ depending only on $n$ and $\mu$.
Therefore there exist $k_0$ and $\widetilde{C}$, depending only on $n$ and $\mu$, such that for all $ k\geq k_0$,
\begin{equation*}
\left\{x\in\mathbb{\overline{R}}^n_+:\;x^T D^2\xi(0)x\leq\frac{2 M'}{M}-\widetilde{C}2^{-\frac{3}{2}\tau k}\right\}
          \subset  \left\{\xi<\frac{M'}{M}-\frac{C}{M^{1/2}}\right\}
\end{equation*}
and
\begin{equation*}
 \left\{\xi<\frac{M'}{M}+\frac{C}{M^{1/2}}\right\}\subset
                   \left\{x\in\mathbb{\overline{R}}^n_+:\;x^T D^2\xi(0)x\leq\frac{2 M'}{M}+\widetilde{C}2^{-\frac{3}{2}\tau k}\right\}.
\end{equation*}
This gives (\ref{SZ_SECT_CONTAINING}).
\end{proof}

\begin{lemma}\label{Lm_T_is_upper_triangular}
  Let $k_0$ and $\tau$ be given by Lemma \ref{Lm_Level_set_A_M}. Then there exists a real invertible  bounded upper-triangular matrix $T$ with $\det T=1$ such that $|T|\leq C$ and
\begin{equation}\label{SZ_TS_M_shape}
   \left(1-CM'^{-\frac{1}{2}\tau }\right){\sqrt{2{M'}}}\overline{B}_1^+
         \subset TS_{{M'}}(u)\subset
   \left(1+CM'^{-\frac{1}{2}\tau }\right){\sqrt{2{M'}}}\overline{B}_1^+
\end{equation}
for all $M'\geq2^{k_0}$, where $C$ depends only on $n$ and $\mu$.
\end{lemma}

\begin{proof}
For $k\geq k_0$, let $M=2^{(1+\tau)k}$ and $M'\in [2^{k-1}, 2^k]$.

By LU decomposition for symmetric positive definite matrices, there exists a unique upper-triangular matrix $T_k$
with real positive diagonal entries such that $[D^2\xi(0)]=T_k^TT_k$.
Obviously, $\det T_k=1$ and $|T_k|\leq C$ for some constant $C$ depending only on $n$ and $\mu$.
Recall (\ref{SZ_DF_EM}) and then
\begin{equation*}
  E_M=T_k^{-1}\overline{B}_1^+.
\end{equation*}
From (\ref{SZ_SECT_CONTAINING}), it follows that
\begin{align*}
     \left(\frac{2 M'}{M}-C2^{-\frac{3}{2}\tau k}\right)^{1/2}T_k^{-1}\overline{B}_1^+
             \subset \frac{S_{M'}(u)}{M^{1/2}} \subset
     \left(\frac{2 M'}{M}+C2^{-\frac{3}{2}\tau k}\right)^{1/2}T_k^{-1}\overline{B}_1^+
\end{align*}
or
\begin{equation}\label{SZ_TK_set}
     \left(1-C2^{-\frac{1}{2}\tau k}\right){\sqrt{2M'}}\overline{B}_1^+
               \subset T_kS_{M'}(u)\subset
     \left(1+C2^{-\frac{1}{2}\tau k}\right){\sqrt{2M'}}\overline{B}_1^+.
\end{equation}

Particularly,
\begin{equation*}
 \left(1-C2^{-\frac{1}{2}\tau {k}}\right){\sqrt{2^{k}}}\overline{B}_1^+
             \subset T_kS_{2^{k-1}}(u)\subset
       \left(1+C2^{-\frac{1}{2}\tau {k}}\right){\sqrt{2^{k}}}\overline{B}_1^+
\end{equation*}
and
\begin{equation*}
 \left(1-C2^{-\frac{1}{2}\tau {(k-1)}}\right){\sqrt{2^{k}}}\overline{B}_1^+
            \subset T_{k-1}S_{2^{k-1}}(u)\subset
      \left(1+C2^{-\frac{1}{2}\tau {(k-1)}}\right){\sqrt{2^{k}}}\overline{B}_1^+.
\end{equation*}
Therefore
\begin{equation*}
       \left(1-C2^{-\frac{1}{2}\tau k}\right)\overline{B}_1^+
              \subset T_kT_{k-1}^{-1}\overline{B}_1^+\subset
       \left(1+C2^{-\frac{1}{2}\tau k}\right)\overline{B}_1^+.
\end{equation*}
for some constant $C$ depending only on $n$ and $\mu$.
By reflection,
\begin{equation*}
       \left(1-C2^{-\frac{1}{2}\tau k}\right)\overline{B}_1
              \subset T_kT_{k-1}^{-1}\overline{B}_1\subset
       \left(1+C2^{-\frac{1}{2}\tau k}\right)\overline{B}_1.
\end{equation*}

It is clear that $T_kT_{k-1}^{-1}$ is upper-triangular. Then by
Lemma A.5 in \cite{Caffarelli-Liyanyan-2003-CommPureApplMath},
\[
    ||T_kT_{k-1}^{-1}-I||\leq C2^{-\frac{1}{2}\tau k}
\]
and then (since $T_k$ is uniformly bounded,)
\[
    ||T_k-T_{k-1}||\leq C2^{-\frac{1}{2}\tau k},
\]
where $C$ depends only on $n$ and $\mu$.
Therefore there exists a unique bounded invertible upper-triangular matrix $T$ with $\det T=1$ such that
\[
    ||T_k-T||\leq C2^{-\frac{1}{2}\tau k}
\]
Combining it with (\ref{SZ_TK_set}) and (\ref{SZ_u_is_quadra_increasing_outside_B^+_1}), we obtain (\ref{SZ_TS_M_shape}).
\end{proof}

\begin{lemma}\label{LM_less_than_2_order}
Let $v(x)=u(y)$ and $y=T^{-1}x$ for $x\in \mathbb{R}_+^n$ and then $v$ solves
\begin{equation}\label{EQ-v-eq-bound}
    \left\{
        \begin{aligned}
          &\det D^2v=1 \quad\;\mbox{in } \mathbb{R}^n_+\backslash  T\Omega_0, \\
          &  v(x)=\frac{1}{2}|x|^2 \quad\mbox{on }  \{x_n=0\}
          \end{aligned}
    \right.
\end{equation}
and for some $C$ depending only on $n$ and $\mu$,
\begin{equation}\label{SZ_est_of_v-x^2/2}
     \left|v(x)-\frac{1}{2}|x|^2 \right|\leq C|x|^{2-\tau}
     \quad in~ \mathbb{\overline{R}}^n_+ \cap\{|x|\geq 2^{k_0}\},
\end{equation}
 where $\tau=\frac{1}{10}$, $k_0$ and $T$ are given by Lemma \ref{Lm_T_is_upper_triangular}.
\end{lemma}

\begin{proof}
(\ref{SZ_TS_M_shape}) implies that
\[
    \left(1-C{M}^{-\frac{1}{2}\tau }\right){\sqrt{2M}}\overline{B}_1^+
            \subset S_{M}(v) \subset
    \left(1+C{M}^{-\frac{1}{2}\tau}\right){\sqrt{2M}}\overline{B}_1^+
\]
and then (\ref{SZ_est_of_v-x^2/2}) holds.

By $\det D^2u(y)=1$ in $\mathbb{R}^n_+\backslash \Omega_0$
and $\det T=1$, we have
\begin{equation*}\label{Eq-v-out}
    \det D^2v(x)=1\quad \mbox{in }\mathbb{R}^n_+\backslash  T\Omega_0.
\end{equation*}
Since $ u(y)=\frac{1}{2}|y|^2$ on $\{y_n=0\}$ and $T$ is upper-triangular,
we deduce
\begin{equation*}\label{SZ_V_Bodry}
    v(x)=\frac{1}{2}|T^{-1}x|^2
    \quad \mbox{on }\{x_n=0\}.
\end{equation*}
By (\ref{SZ_est_of_v-x^2/2}),
\[
    \left|\frac{1}{2}|T^{-1}Mx|^2-\frac{1}{2}|Mx|^2\right|
    \leq C|Mx|^{2-\tau} \quad \mbox{on }\{x_n=0\}.
\]
Let $M\rightarrow\infty$ and we conclude that
\[
    |T^{-1}x|^2=|x|^2 \quad \mbox{on }\{x_n=0\}.
\]
Therefore
$v(x',0)=\frac{1}{2}|x'|^2$.
\end{proof}

\subsection{Linear approach}

In this subsection we prove the following lemma that completes the proof of Theorem 1.1, where the results established in Section 2 and Section 3 will be used essentially.

\begin{lemma}\label{Lm_BOOOOOOOOOOOOOOOSS}
Let $v$ be given by Lemma \ref{LM_less_than_2_order}.
Then there exists some constant $b_n$ such that
\begin{equation}\label{SZ_lemaBoss}
    \left|v(x)-\frac{1}{2}|x|^2-b_nx_n\right|
    \leq C\frac{x_n}{|x|^{n}}
    \quad\mbox{in }\mathbb{\overline{R}}^n_+\backslash B^+_R,
\end{equation}
where  $C$ and $R$ depend only on $n$ and $\mu$. Furthermore, for any $k\geq1$,
\begin{equation}\label{SZ_lemaBoss2}
    |x|^{n-1-k}\left|D^k\left(v(x)-\frac{1}{2}|x|^2-b_nx_n\right)\right|
    \leq C
    \quad \mbox{in }\mathbb{\overline{R}}^n_+\backslash B^+_R,
\end{equation}
where $C$ also depends on $k$.
\end{lemma}

\begin{proof}
Since $T$ is bounded,
there exists $R_1>0$ depending only on $n$ and $\mu$ such that
$T\Omega_0\subset \overline{B}^+_{R_1}$.
Let ${V}(x)=v(x)-\frac{1}{2}|x|^2$. By (\ref{SZ_est_of_v-x^2/2}) and
Corollary \ref{LM_D^kv_est},
\begin{equation}\label{SZ_ES_D2V}
     |DV(x)|\leq C|x|^{1-\tau}
     \quad\mbox{and}\quad
     |D^2 V(x)|\leq C|x|^{-\tau}\quad
     \mbox{in }\mathbb{\overline{R}}^n_+\backslash B^+_{R_1},
\end{equation}
where $\tau=\frac{1}{10}$ and $C$ depends only on $n$ and $\mu$.

Differentiating $\ln\det(I_n+D^2{V})=0$ with respect to $x_k$, $k=1,\cdots,n$,
we get
\begin{equation*}
     {a}_{ij}(x)D_{ij}V_k(x)=0
     \quad\mbox{in }\mathbb{R}^n_+\backslash B^+_{R_1},
\end{equation*}
where ${a}_{ij}(x)=[D^2 V+I_n]^{ij}(x)\mbox{ and } V_k=D_kV.$

In view of \eqref{SZ_ES_D2V}, we obtain that
\begin{equation*}
   |{a}_{ij}(x)-\delta_{ij}|\leq C|x|^{-\tau}
     \quad\mbox{in }\mathbb{\overline{R}}^n_+\backslash B^+_{R_1}
\end{equation*}
for some constant $C$ depending only on $n$ and $\mu$ and for $k=1,\cdots,n$,
\begin{equation*}
  |DV_k(x)|\rightarrow0\quad \mbox{as } |x|\rightarrow \infty.
\end{equation*}

Observe $V(x',0)=0$ and then $V_k(x',0)=0$ for $k=1,\cdots,n-1$ and $|x'|>R_1$.
By Theorem \ref{TM_Asymp_Behr_of_solu_of_Linear_Eq}, we have
for any $k=1,\cdots,n-1$,
\begin{equation*}
      |V_k(x)|\leq C\frac{x_n}{|x|^{n}}
      \quad\mbox{in }\mathbb{\overline{R}}^n_+\backslash B^+_{R_1},
\end{equation*}
where $C$ depends only on $n$ and $\mu$.
It follows that
for any $k=1,\cdots,n-1$,
\[
     |V_{kn}(x',0)|\leq\frac{C}{|x'|^{n}},
     \quad |x'|\geq R_1.
\]
Since $n\geq2$,
there exists some $b_n$ such that
\begin{equation*}
        V_n(x',0)\rightarrow b_n
        \quad\mbox{as }|x'|\rightarrow \infty.
\end{equation*}
By Theorem \ref{Co_Limits_of_solutions_of_linear_Eq},
we obtain that
\begin{equation}\label{SZ_Vn_tendsto_bn}
       V_n(x)\rightarrow b_n\quad\mbox{as } |x|\rightarrow\infty.
\end{equation}

From $\ln\det(I_n+D^2{V})=\ln\det I_n=0$, we deduce
\begin{equation*}
 \left\{
        \begin{aligned}
          &\widetilde{a}_{ij}(x)D_{ij}(V-b_nx_n)=0
       \quad\mbox{in } \mathbb{R}^n_+\backslash B_{R_1}^+, \\
          &  V-b_nx_n=0 \quad\quad\quad\quad\quad\;\;\mbox{on }  \{x_n=0\},
          \end{aligned}
    \right.
\end{equation*}
where ${\widetilde{a}}_{ij}(x)=\int_0^1[sD^2 V+I_n]^{ij}(x)ds$ and in view of
(\ref{SZ_ES_D2V}),
\begin{equation*}
   |{\widetilde{a}}_{ij}(x)-\delta_{ij}|\leq C|x|^{-\tau}
     \quad\mbox{in }\mathbb{\overline{R}}^n_+\backslash B^+_{R_1}
\end{equation*}
for some constant $C$ depending only on $n$ and $\mu$.
Since, by \eqref{SZ_Vn_tendsto_bn},
\begin{equation*}\label{SZ-Dv-bn-sml}
  |D(V(x)-b_nx_n)|\rightarrow 0\quad\mbox{as }|x|\rightarrow\infty,
\end{equation*}
we have, by Theorem \ref{TM_Asymp_Behr_of_solu_of_Linear_Eq},
there exist $R\geq R_1$ and $C$ depending only on $n$ and $\mu$
such that \eqref{SZ_lemaBoss} holds. And then
applying Corollary \ref{LM_D^kv_est} with $w=V-b_nx_n$ and $\gamma=n-1$,
we have \eqref{SZ_lemaBoss2}.
\end{proof}

Finally, Theorem \ref{CoT2infty} follows
from Lemma \ref{Lm_BOOOOOOOOOOOOOOOSS} immediately.

\section{Proofs of Corollary \ref{Co_f=1} and Theorem \ref{TM_Main_TM_extence}}

\emph{Proof of Corollary \ref{Co_f=1}.}
By Theorem \ref{TM_Main_TM}, there exist some symmetric positive definite matrix $A$ with $\det A=1$,
vector $b\in \mathbb{R}^n$ and constant $c\in\mathbb{R}$
such that
\[
   E(x):=u(x)-\frac{1}{2}x^T Ax-b\cdot x-c\rightarrow 0
    \quad \mbox{as }|x|\rightarrow \infty.
\]
and
$
    E=0 $ on $\{x_n=0\}.
$
Furthermore, from $\det D^2u-\det A=\det D^2u-1=0$
and $[D^2u]=[A+D^2E]$, it can be deduced that
\[
    a_{ij}D_{ij}E=0\quad\mbox{in }\mathbb{R}^n_+,
\]
where $a_{ij}(x)=\int_0^1[sD^2 u+(1-s)A]^{ij}(x)ds.$

By the maximum principle, we get $E(x)\equiv0 $,
i.e., $u(x)=\frac{1}{2}x^T Ax+b\cdot x+c.$
$\hfill\Box$

\bigskip

Before proving Theorem \ref{TM_Main_TM_extence}, we define
\begin{equation*}
   u_{\pm}(x',x_n)= \frac{1}{2}|x'|^2+\int_{0}^{x_n}\int_{0}^{t}f_{\pm}(s)dsdt,\quad (x',x_n)\in \mathbb{R}^n_+,
\end{equation*}
where $f_{\pm}(s)$ satisfy
\begin{equation*}
 support(f_{\pm}-1)\subset[0,1]\quad\mbox{and}
\quad0\leq f_+(s)\leq1\leq f_-(s)\leq\Lambda,
\quad s\geq0.
\end{equation*}
Then $u_{\pm}\in C^{1,1}(\overline{\mathbb{R}}^n_+)$ are convex functions satisfying
\begin{equation*}
\left\{
   \begin{aligned}
        &\det D^2u_{\pm}(x',x_n)=f_{\pm}(x_n)\quad\mbox{in } \mathbb{R}^n_+,\\
        &u_{\pm}=\frac{1}{2}|x|^2\quad\quad\quad\quad\quad\quad\quad\;\;\mbox{on }  \{x_n=0\},
    \end{aligned}
  \right.
\end{equation*}
and
\begin{equation}\label{SZ-u-supb-bound}
\begin{split}
  &\frac{1}{2}|x|^2-x_n\leq
   u_+\leq \frac{1}{2}|x|^2\leq u_-\leq
   \frac{1}{2}|x|^2+(\Lambda-1)x_n
   \quad\mbox{in } \mathbb{R}^n_+.
\end{split}
\end{equation}

\bigskip

\emph{Proof of Theorem \ref{TM_Main_TM_extence}.}
The uniqueness of solutions in Theorem \ref{TM_Main_TM_extence}
can be deduced from the comparison principle.
As for the existence part, we only need to show it
under additional hypothesis  $\Omega_0\subset \overline{B}_{1/2}^+$, $A=I_n$.
In fact, by LU decomposition for symmetric positive definite matrices,
there exists a unique upper-triangular matrix $Q$ with real positive diagonal entries
such that $Q^TQ=A$ and $\det Q=1$.
Then the existence of $u$ satisfying
\eqref{SZ_Main_TM_Equation} and \eqref{SZ-ext-tend} as (1.2) holds
is equivalent to the existence of $w$ satisfying
\begin{equation*}
       \left\{
           \begin{aligned}
               &\det D^2w=\widetilde{f}  \quad\quad\quad\quad\quad\quad\quad\quad\quad\quad\quad\;\,\quad\quad
               \mbox{in } \mathbb{R}^n_+,\\
               &w=\frac{1}{2}|x|^2+b\cdot x+c\quad\quad\quad\quad\quad\quad\quad\quad\quad\quad
               \mbox{on }  \{x_n=0\},\\
               &\lim_{|x|\rightarrow\infty}\left|w(x)-\left(\frac{1}{2}|x|^2+b\cdot x+c\right)\right|=0
            \end{aligned}
       \right.
\end{equation*}
as
$\mbox{support}(\widetilde{f}-1)\subset \overline{B}_{1/2}^+$,
by setting
\[
       u(x)=M^2w\left(\frac{Qx}{M}\right)~\mbox{and}~
       f(x)=\widetilde{f}\left(\frac{Qx}{M}\right)
\]
for $M$ large enough.

\bigskip

Next we prove Theorem \ref{TM_Main_TM_extence}
with $\Omega_0\subset \overline{B}_{1/2}^+$ and $A=I_n$.

For any $R>1$, let $u_R(x)$
be the unique convex viscosity solution of
\begin{equation}\label{EQ_uR}
     \left\{
           \begin{aligned}
             &\det D^2u_R=f \quad \mbox{in } B_R^+,\\
             &u_R=\frac{1}{2}|x|^2\quad\quad\mbox{on }\partial B_R^+.\\
          \end{aligned}
    \right.
\end{equation}

Suppose $0\leq f\leq \Lambda$ for some $\Lambda>1$. Since $\mbox{support}(\widetilde{f}-1)\subset \overline{B}_{1/2}^+$ and, by \eqref{SZ-u-supb-bound},
\begin{equation*}\label{SZ-uneg}
  u_--(\Lambda-1)x_n
      \leq\frac{1}{2}|x|^2\leq
u_++x_n\quad\mbox{in } \mathbb{R}^n_+,
\end{equation*}
we deduce, by the comparison principle,
\[
u_--(\Lambda-1)x_n\leq u_R\leq u_++x_n\quad\mbox{in }B_R^+.
\]
By \eqref{SZ-u-supb-bound} again,
\begin{equation}\label{j1}
  \frac{1}{2}|x|^2-(\Lambda-1)x_n\leq u_R \leq\frac{1}{2}|x|^2+x_n\quad\mbox{in }B_R^+.
\end{equation}
It follows that
\begin{equation}\label{Sz-u-good-bound-ww}
  -(\Lambda-1)\leq D_n u_R \leq1 \quad\mbox{on } \{x_n=0\}\cap \overline{B}_R^+.
\end{equation}

Next we fix any $\widetilde{R}>1$ (large enough) and assume that $R\geq2\widetilde{R}$. From (\ref{Sz-u-good-bound-ww}) and the boundary condition in (\ref{EQ_uR}), we see that
$$|Du_R|\leq \max\{\Lambda-1,1, 2\widetilde{R}\}\quad \mbox{on }\left\{x_n=0,~|x'|\leq2\widetilde{R}\right\},$$
and from (\ref{j1}) and the convexity of $u_R$, we deduce

\begin{align*}
\sup_{x_n>0,|x|=\widetilde{R}}|Du_R|
&\leq\max\left\{ \max_{x_n=0,|x'|\leq2\widetilde{R}}|Du_R|,
\frac{\sup_{x_n>0,|x|=2\widetilde{R}}u_R-\inf_{x_n>0,|x|=\widetilde{R}}u_R}{\widetilde{R}}
\right\}\\
&\leq C
\end{align*}
for some constant $C$ depending on $\widetilde R$.
Therefore by the convexity of $u_R$ again, we conclude that $|Du_R|$ is bounded on $\overline{B}_{\widetilde R}^+$.
Thus, along a sequence $R_j\rightarrow \infty$,
$u_{R_j}\rightarrow\widetilde u_{\widetilde R}\in C^0(B^+_{\widetilde R})$. By diagonal arguments, there exist a subsequence of $\{u_{R_j}\}$ (still denoted by $\{u_{R_j}\}$) and $\check u\in C^0(\mathbb{R}^n_+)$ such that
$$u_{R_j}\rightarrow\check u\quad\mbox{in }C^0_{loc}(\mathbb{\overline{R}}^n_+).$$

Moreover, by (5.2), $\check u$ is a convex viscosity solution of
\begin{equation*}
     \left\{
           \begin{aligned}
             &\det D^2\check u=f \quad \mbox{in } \mathbb{R}^n_+,\\
             &\check u=\frac{1}{2}|x|^2\quad\quad\mbox{on }\{x_n=0\}\\
          \end{aligned}
    \right.
\end{equation*}
and, by (5.3),
\begin{equation*}
\frac{1}{2}|x|^2-(\Lambda-1)x_n\leq\check u \leq\frac{1}{2}|x|^2+x_n
\quad\mbox{in } \mathbb{R}^n_+.
\end{equation*}
By Theorem \ref{TM_Main_TM}, we obtain the existence of $u$ by adding a suitable affine function to $\check u$.
$\hfill\Box$



\end{document}